\documentclass[11pt,reqno]{amsart}
\usepackage{enumerate}
\usepackage{a4wide}
\usepackage{amsmath}
\usepackage{amsfonts}
\usepackage{amssymb}
\usepackage{latexsym}
\usepackage{mathrsfs}
\usepackage{amsthm}
\usepackage[active]{srcltx}
\usepackage{tabularx}
\usepackage{graphicx}
\usepackage{bbm}
\usepackage{setspace}
\usepackage{hyperref}
\usepackage{geometry}
\usepackage[usenames,dvipsnames]{color}
\everymath{\displaystyle}
\usepackage{url}
\usepackage{hyperref}

\newcommand{\Z}{\mathbb Z}
\newcommand{\Q}{\mathbb Q}
\newcommand{\A}{\mathbb A}
\newcommand{\Aa}{\mathcal A}

\newcommand{\F}{\mathbb F}

\newcommand{\C}{\mathbb C}

\newcommand{\On}{(\Om/\id{p}^n)^\times}
\DeclareMathOperator{\GL}{GL}
\DeclareMathOperator{\val}{val}

\DeclareMathOperator{\Disc}{Disc}
\DeclareMathOperator{\Norm}{Norm}
\DeclareMathOperator{\GSp}{GSp}
\DeclareMathOperator{\SL}{SL}

\DeclareMathOperator{\Gal}{Gal}
\DeclareMathOperator{\Ind}{Ind}
\DeclareMathOperator{\PGL}{PGL}

\DeclareMathOperator{\Aut}{Aut}

\DeclareMathOperator{\Trace}{Tr}

\DeclareMathOperator{\cond}{cond}
\DeclareMathOperator{\sw}{sw}

\def\id#1{{\mathfrak{#1}}}      

\DeclareMathOperator{\NCM}{\bf NCM}

\DeclareMathOperator{\LO}{{\bf LO}}

\DeclareMathOperator{\LT}{{\bf LT}}

\DeclareMathOperator{\St}{St}
\DeclareMathOperator{\BC}{Ind_{W(E)}^{W(\Q_p)}(\theta)}
\def\Om{\mathscr{O}}

\theoremstyle{plain}
\newtheorem{thm}{Theorem}
\newtheorem{lemma}[thm]{Lemma}

\newtheorem{prop}[thm]{Proposition}
\newtheorem{question}[thm]{Question}
\theoremstyle{definition}
\newtheorem{dfn}[thm]{Definition}
\theoremstyle{remark}
\newtheorem{remark}[thm]{Remark}
\newtheorem{example}{Example}
\begin{document}

\title[On the number of Galois orbits of newforms]{On the number of Galois orbits of newforms}

\author{Luis Dieulefait}
\address{Facultat de Mathematiques, Universitat de Barcelona}
                     \email{ldieulefait@ub.edu}

\author{Ariel Pacetti}
\address{FaMAF-CIEM, Universidad Nacional de C\'ordoba. C.P:5000,  C\'ordoba, Argentina.}
\email{apacetti@famaf.unc.edu.ar}
\thanks{AP was partially supported by PIP 2014-2016 11220130100073}

\author{Panagiotis Tsaknias}
\email{p.tsaknias@gmail.com}
\thanks{PT was partially funded by by the Luxembourg Research Fund
INTER/DFG/12/10/COMFGREP in the framework of the priority program 1489
of the Deutsche Forschungsgemeinschaft}

\keywords{Hecke Galois orbits}
\subjclass[2010]{Primary: 11F03, Secondary: 11F11}

\begin{abstract}

  Counting the number of Galois orbits of newforms in
  $S_k(\Gamma_0(N))$ and giving some arithmetic sense to this number
  is an interesting open problem. The case $N=1$ corresponds to
  Maeda's conjecture (still an open problem) and the expected number
  of orbits in this case is 1, for any $k \ge 16$. In this article we
  give local invariants of Galois orbits of newforms for general $N$
  and count their number. Using an existence result of newforms with
  prescribed local invariants we prove a lower bound for the number of
  non-CM Galois orbits of newforms for $\Gamma_0(N)$ for large enough
  weight $k$ (under some technical assumptions on $N$). Numerical
  evidence suggests that in most cases this lower bound is indeed an
  equality, thus we leave as a Question the possibility that a
  generalization of Maeda's conjecture could follow from our work.  We
  finish the paper with some natural generalizations of the problem
  and show some of the implications that a generalization of Maeda's
  conjecture has.
\end{abstract}

\maketitle

\tableofcontents

\section*{Introduction}

A conjecture of Maeda predicts that there is a unique Galois orbit of
level $1$ newforms for all weights $k \geq 16$. A natural problem is
to study what happens while working with modular forms of arbitrary
level $N$. For small weights, the number of Galois orbits in
$S_k(\Gamma_0(N)$ is hard to understand, for example in weight $2$
(which is not in the original Maeda's conjecture) there are many
elliptic curves of the same conductor $N$. However, while computing
spaces of modular forms of a fixed level and varying the weight $k$, the
situation changes completely. Surprisingly, the number of orbits tends
to stabilize very fast, and the numbers obtained follow some pattern
(see for example the data in \cite{Tsaknias2012a}).

While proving Maeda's conjecture of newforms for $\SL_2(\Z)$ is a very
hard problem, it is fairly easy to prove the lower bound $1$ for the
number of Galois orbits when $k \ge 16$, which corresponds to the ``easy''
inequality. The purpose of the present article is to present
invariants of Galois orbits of eigenforms, and use them to give a
lower bound for the number of Galois orbits of newforms in
$S_k(\Gamma_0(N))$ for $k$ large enough (i.e. for all $k \ge k_0$, for some $k_0\geq2$). In
many instances, the numerical data seems to indicate that such
inequality is in fact an equality.

The invariants introduced are of two different natures: a local one,
namely the Galois orbit of the \emph{local type} of the automorphic
representation at each prime dividing $N$; and a local-global one,
coming from the Atkin-Lehner eigenvalue at $p$ of the modular form
$f$. Recall that the local type can be thought of (via the
Local-Langlands correspondence) as the isomorphism class of the
restriction of the Weil-Deligne representation to the inertia subgroup
(see Section~\ref{sec:counting}). The Atkin-Lehner sign is more
subtle, and it is not clear how to obtain it from the Weil-Deligne
representation.

The lower bound we prove is of the following form. Let $\NCM(N,k)$
denote the number of Galois orbits of non-CM newforms of level $N$ and
weight $k$. If $N$ is a prime power or if $N$ is square-free, then
\begin{equation}
  \label{eq:mainformula}
\prod_{q \mid N} \LO(q^{\val_q(N)}) \le \NCM(N,k),  
\end{equation}
for all $k$ large enough, where the values of $\LO(q^r)$ are given in
Theorem~\ref{thm:valuesofLO}. Let us
explain a little bit all the ingredients of the formula and its proof. 

In Section~\ref{sec:counting}, we recall the theory of local types for
$\GL_2$, and consider Galois conjugacy classes of them. Since we want
to count the number of Galois orbits of modular forms, a naive idea is
that while conjugating a modular form $f$, one also conjugates the
local types, hence while identifying global conjugates one should do
the same locally. The section contains a detailed description of local
types and their number, the main result being a formula for the number
of Galois orbits of local types of level $p^n$ for any prime $p$ (the
case $p=2$ being the hardest one).

Section~\ref{sec:modformtypes} considers local types coming from
modular forms. There are two advantages on doing so: first we prove
(see Lemma~\ref{lemma:coeffppalseries} and
Lemma~\ref{lemma:coeffsupercuspidal}) that if a modular form $f$ has a
local type $\tilde{\tau}$, then its coefficient field is an extension
of $\Q$ with enough endomorphisms. In particular, this shows that the
naive approach (looking at local Galois orbits) is correct in most
instances. This is not true in general, but it is true under the hypothesis
on $N$ stated before, i.e. $N$ is a prime power or a square free
integer (see Remark~\ref{remark:localtypegeneralN} to understand the
general case). The second advantage of working with modular forms of trivial Nebentypus is
that we have the theory of Atkin-Lehner involutions. Clearly their
eigenvalues are constant on Galois orbits (see
Lemma~\ref{lemma:galoisinvariance}), thus they give an extra
invariant. There is an interesting phenomenon while computing
Atkin-Lehner eigenvalues: a modular form of level $p$ (prime) might
have any Atkin-Lehner eigenvalue (for different values of $p$ and
$k$ both are attained) but its twist by the quadratic character
unramified outside $p$ does not! Then we might have two different
Galois orbits of level $p$ (distinguished by the Atkin-Lehner
eigenvalue) whose twists (of level $p^2$) still give two different
orbits, but both of them having the same Atkin-Lehner eigenvalue. This
phenomenon suggests that we do not have to consider the Atkin-Lehner
sign as an invariant, but what we call the \emph{minimal Atkin-Lehner
  sign} (see Definition~\ref{defi:minimalAL}).

An important result in this direction is the determination of what are the
possible Atkin-Lehner signs for each local type. Such description is
given in Theorem~\ref{thm:AL}, which describes when the local type
determines the minimal Atkin-Lehner sign uniquely, and when it does
not. For the latter, we prove that the local sign varies while
twisting by the unramified quadratic character at $p$. Then we can
count the number of pairs $(\tilde{\tau},\epsilon)$ consisting of an
isomorphism class of local types of level $p^n$ and its compatible
minimal Atkin-Lehner sign. This number is denoted by $\LO(p^n)$ and is
the one appearing in ~(\ref{eq:mainformula}). An important result in
this section is a precise formula for such value (see
Theorem~\ref{thm:valuesofLO}).

Section~\ref{section:existence} considers the problem of the existence
of pairs $(\tilde{\tau},\epsilon)$ as before, for large values of
$k$. The main result is Theorem~\ref{thm:typesexistence}, in the case
$N$ a prime power or square-free. The proof is based on results from
Weinstein (\cite{Weinstein2009}) and Kim-Shin-Templier
\cite{Templier}. The latter article proves an existence result of
modular forms with a fixed local representation at $p$ (not being
principal series), not just its type!. Such a result is very strong,
but it implies Theorem~\ref{thm:typesexistence} under our
hypothesis. For general $N$, a different approach must be taken, as
principal series would need to be included (see
Remark~\ref{remark:localtypegeneralN}). We want to stress that if
Theorem~\ref{thm:typesexistence} holds for general $N$, then
(\ref{eq:mainformula}) holds in general (since the restriction on $N$
is only used in such result).

It is natural to ask why we discard the CM modular forms in our
result. The reason is twofold: first of all, modular forms with
complex multiplication do form an orbit on their own. The second one
is that (for $k$ large enough) when the space of newforms of a given
level $N$ contains a CM Galois orbit, there is another Galois orbit
with the same local type without complex multiplication. 
\begin{example}
Let $N=9$ and $k=16$. This space contains a unique modular form with
complex multiplication, whose $q$-expansion starts
$q - 32768 q^4 + 1244900 q^7 +O(q^{12})$. The local characters giving
the local representation can be computed with \cite{sage}; they
correspond to the character over the unramified quadratic extension
of $\Q_3$, sending a generator $s$ of $\F_9^\times$ to
$\sqrt{-1}$. There is another form, with $q$-expansion
$q + a q^2 + 87112 q^4 + 464 a q^5 - 2591260 q^7 + 54344 a q^8
+O(q^{10})$, where $a^2 = 119880$ whose characters (at inertia) are
exactly the same hence both representations have the same local
type. Note that the latter form does not have complex multiplication
(as the $5$-th coefficient is non-zero).  
\end{example}
This same situation holds in general and is part of
Theorem~\ref{thm:typesexistence}, whose proof uses the fact that the
number of non-CM forms with prescribed local types grows linearly on
the weight $k$, while the number of CM forms is constant. With all
these ingredients, the proof of the stated bound
(Theorem~\ref{cor:NCMlowbound}) is straightforward.

In \cite{Tsaknias2012a} the author proposed a generalization of
Maeda's conjecture (Conjecture 2.2) to arbitrary levels $N$ as
follows:
\begin{itemize}
\item the function $\NCM(N,k)$ is constant in the variable $k$ for $k$
  large enough
\item the limit function $\NCM(N):= \lim_{k \to \infty}\NCM(N,k)$ is
  multiplicative.
\item some values of $\NCM(p^n)$ were tabulated based on numerical experiments.
\end{itemize}

The present article started from the effort to prove that the
tabulated numbers have some meaning, and to express them as Galois orbits
invariants. While doing so, we realized that we do not expect the
function $\NCM(N)$ to be multiplicative (see
Remark~\ref{remark:localtypegeneralN}). The reason is that the
automorphisms of the coefficient field are not enough in general to
conjugate two different local types independently. Examples for this
involve huge levels which are nowadays unfeasible to compute with
nowadays resources (this was probably the reason why this phenomena
went unobserved).

We end the article with some possible generalizations of the present
ideas, and some applications. We propose a question
(Question~\ref{conecture:maeda}) which is in the spirit of Maeda's
original conjecture. Numerical evidence (gathered by the third named
author) suggests that in most of the considered cases this lower bound
is indeed an equality (for large enough weight $k$) to the number of
such Galois orbits, thus we leave as a Question the possibility that a
generalization of Maeda's conjecture could follow from our work; in
which case, for historical reasons, it should be called the
``Maeda-Tsaknias'' conjecture.  In Example~\ref{example:discrepancy} we
present a discrepancy between the experimental values of
$\NCM(256,12)$ and our lower bound which seems to persist for all
weights greater than $12$. We could not find any extra invariant that
justifies this discrepancy (it is an interesting problem to
investigate). In particular, if the value of $\NCM(256)$ is indeed
$12$, Question~\ref{conecture:maeda} needs to be reformulated taking
into account the missing invariants.

\subsection*{Acknowledgements} The authors would like to thank Kimball
Martin, Michael Harris and David Roberts for many useful
conversations. The third author would also like to thank Gabor Wiese
for many helpful conversations and remarks during the earlier stages
of this article.

\section{Inertial types for $\GL_2$}
\label{sec:counting}

Let $\Aa_p$ denote the set of isomorphism classes of complex-valued
irreducible admissible representations of $\GL_2(\Q_p)$. The local
Langlands correspondence gives a bijection between $\Aa_p$ and the
isomorphism classes of two-dimensional Frobenius-semisimple
Weil-Deligne representations of $\Q_p$, say
$\pi \leftrightarrow \tau(\pi)$. Furthermore, the equivalence
preserves $L$-functions and $\epsilon$-factors (see \cite{Kutzko} and
\cite{Bushnell06}). Via the local-Langlands correspondence, we will
move to-and-from $\Aa_p$ indistinctly.

\begin{dfn}
  A local inertial type of a Weil-Deligne representation $\tau$ is the
  isomorphism class of its restriction to the inertia subgroup. We
  denote it by $\tilde{\tau}$.  We say that a type is \emph{trivial}
  or \emph{unramified} if $\tilde{\tau}$ is the trivial
  representation.
\end{dfn}

\begin{remark}
  The inertial type can also be described in terms of the restriction $\pi|_{\GL_2(\Z_p)}$, as explained in \cite{Henniart}. See also\cite [Section 2.1]{Weinstein2009}.
\end{remark}

While working with local types, the maximal ideal is always clear from
the context. For this reason, and to ease notation, for the rest of
the article we will use the term \emph{conductor} (of a
representation, of a character, etc) to denote the exponent of the
conductor. We hope this will not create any confusion.

\begin{dfn}
  A global inertial type is a collection $(\tilde{\tau_p})_p$ with $p$ running over all
  prime numbers, where each $\tilde{\tau_p}$ is a local type at $p$
  and $\tilde{\tau_p}$ is trivial for all primes but finitely
  many.
\end{dfn}

\begin{thm}
  Any element $\pi$ of $\Aa_p$ is one of the following:

  \medskip

  \noindent $\bullet$ {\bf Principal series:} given characters
  $\chi_1,\chi_2:\Q_p^*\to\C^*$ such that
  $\chi_1\chi_2^{-1}\neq |\ |^{\pm1}$, the representation
  $\pi(\chi_1,\chi_2)$ is the induction of a $1$-dimensional
  representation of the Borel subgroup of $\GL_2(\Q_p)$, with action
  given by $\chi_1 \otimes \chi_2$.  The central character of
  $\pi(\chi_1,\chi_2)$ equals $\chi_1\chi_2$ and its conductor equals
  $\cond(\chi_1)+\cond(\chi_2)$.

  \smallskip

  \noindent $\bullet$ {\bf Special representations or Steinberg:} if
  $\chi_1 \chi_2^{-1} = | \ |^{\pm1}$, the representation
  $\pi(\chi_1,\chi_2)$ contains an irreducible codimension $1$
  subspace/quotient. Such representations are called \emph{Steinberg}
  and they are twists of a ``primitive'' (or standard) one denoted
  $\St$.
  The central character of $\St \otimes \chi$ equals $\chi^2$ and its
  conductor equals
\[
\cond(\St \otimes \chi)=
\begin{cases}
2\cond(\chi) & \text{ if }\chi \text{ is\ ramified,}\\
1 & \text{ otherwise.}
\end{cases}
\]
\smallskip

\noindent $\bullet$ {\bf Supercuspidal representations}: the remaining
ones, see \cite{KutzkoSCI,KutzkoSCII}. 
\label{thm:characterization}
\end{thm}

Using the previous classification the local Langlands correspondence
is given explicitly by:
  \begin{enumerate}
  \item The Weil-Deligne representation attached to
    $\pi(\chi_1,\chi_2)$ via the local Langlands correspondence
    consists of the pair $(\chi_1 \oplus \chi_2,0)$, i.e. the Weil
    representation is given by the direct sum $\chi_1 \oplus \chi_2$
    (recall that we are identifying characters of the Weil group and
    of $\Q_p^\times$ via local class field theory) and the monodromy
    is trivial.
  \item The Weil-Deligne representation attached to the representation
    $\St \otimes \chi$ consists of the pair
    $\left(\chi \omega_1 \oplus \chi,\left(\begin{smallmatrix} 0 & 1\\
          0 & 0\end{smallmatrix} \right)\right)$, where $\omega_1$ is
    the unramified character giving the action of $W(\Q_p)$ on the
    roots of unity. This is the only case of non-trivial monodromy.
\item If $p \neq 2$, the Weil representation attached to the
  supercuspidal representations via the local Langlands correspondence equals
  $\Ind^{W(E)}_{W(\Q_p)}\theta$, where $E/\Q_p$ is a quadratic
  extension, and $\theta: W(E)\to\C^\times$ is a character.
  Furthermore, regarding $\theta$ as a character of $E^\times$, such
  representation is irreducible precisely when $\theta$ does not
  factor through the norm map $\Norm:E^\times\to\Q_p^\times$.  Let
  $\epsilon_{E}$ denote the quadratic character of $\Q_p^\times$
  associated by local class field theory to the extension
  $E/\Q_p$. The central character of $\BC$ equals
  $\theta|_{\Q_p}\cdot \epsilon_{E}$ and its conductor equals
\[
  \cond(\BC)=\begin{cases}
2\cond(\theta) & \text{ if }E/\Q_p \text{ is unramified},\\
\cond(\theta)+\cond(\epsilon_{E}) & \text{ otherwise.}
\end{cases}
\]
If $p=2$, besides the cases described above, the projective image of the
Weil representation can be one of the sporadic groups $A_4$ or $S_4$
corresponding to the \emph{sporadic supercuspidal representations} (as studied
by Weil in \cite{Weil}), see \ref{sporadicones} for more details.
\end{enumerate}

\begin{remark}
  The image of the inertia subgroup of a Weil representation lies in a
  finite extension of $\Q$, hence it makes sense to look at its Galois
  conjugates.
\end{remark}

\begin{dfn}
  Given $\pi_1,\pi_2 \in \Aa_p$ they have \emph{Galois conjugate local
    inertial type} if there exists $\sigma \in \Aut_\Q(\C)$ such that
  the local inertial type of $\tau(\pi_1)$ and $\sigma(\tau(\pi_2))$
  agree. By a \emph{local type Galois orbit} we mean an equivalence
  class of Galois conjugate local inertial types.
\end{dfn}

\begin{remark}
  Elements in the same local type Galois orbit need not have the same central character.
\end{remark}

\subsection{Counting local type Galois orbits}

Let $p$ be a prime number, and denote by $\LT(p^n)$ the number of
local type Galois orbits of conductor $n$ with trivial Nebentypus. For
$a$ a positive integer, let $\sigma_0(a)$ denote the number of
positive divisors of $a$.

\begin{thm}\label{prop:LOformula}
  Let $p \neq 2$ be a prime number. Then the values of $\LT(p^n)$ are given in table~\ref{table:pneq3}.
\begin{table}[h]
\scalebox{0.95}{
  \begin{tabular}{||l|c|c|c|c||}
  \hline
$n$ & $\text{P.S.}$ & $\St$ & $\text{S.C.U.}$ & $\text{S.C.R}$\\
\hline
  $1$ & --- & $1$ & --- & --- \\
  \hline
  $2$ & $\sigma_0(p-1)-1$ & $1$ & $\sigma_0(p+1)-2$ & --- \\
    \hline
    $\stackrel{p \neq 3}{n \ge 3 \text{ odd}}$ & --- & --- & --- & $2$\\
    \hline
    $\stackrel{p = 3}{n \ge 3 \text{ odd}}$ & --- & --- & --- & $4$\\
    \hline    
    ${n \ge 3}{\text{even}}$ & $\sigma_0(p-1)$ & --- & $\sigma_0(p+1)$ & ---  \\
    \hline
  \end{tabular}}
\caption{Values for $\LT(p^n)$  for $p \neq 2$.}
\label{table:pneq3}
\end{table}
\label{thm:mainthmodd}
\end{thm}

\begin{remark}
  There exists a ramified supercuspidal representation of conductor
  $2$ for $p \equiv 3 \pmod 4$, but its local type matches that of an
  unramified supercuspidal representation (see for example
  \cite[Theorem 2.7]{Gerardin75}), which is why we do not count it in the previous table.
\end{remark}

By Theorem~\ref{thm:characterization}, to compute $\LT(p^n)$ it is
enough to count the number of Galois orbits for the Principal Series,
the Steinberg and the Supercuspidal types. The Steinberg type is the
easy one (they are all twists of $\St$), while the Principal Series
count comes from the well known group structure of
$(\Z_p/p^n)^\times$.

Supercuspidal representations are induced from a character $\theta$ of
a quadratic extension $E$ of $\Q_p$. By
Theorem~\ref{thm:characterization} such induction has trivial
Nebentypus precisely when the restriction of $\theta$ to $\Q_p^\times$
is fixed (and matches that of $\epsilon_E$).  Clearly two induced
representations have Galois conjugate inertial types precisely when
the quadratic field $E$ is the same for both of them, and the two
characters are Galois conjugate. This occurs precisely when one is a
power (prime to the order) of the
other.

Let $E=\Q_p(\sqrt{d})/\Q_p$ be a quadratic extension, let $e$ denote
the ramification degree of $E/\Q_p$, let $\Om_E$ denote the ring of
integers of $E$ and $\id{p}$ its maximal ideal. Let $k$ denote the
residual field $\Om_E/\id{p}$, and $q = p^r = \# k$. For $n$ a
positive integer let $\xi_n$ denote a primitive $n$-th root of unity.

\begin{thm}
  Let $n$ be a positive integer and let $d \in \{\pm 1, \pm3\}$. Then
  the group structure of $(\Om/\id{p}^n)^\times$ is given in Table
  \ref{table:groupstructure} where: -- means no condition, and the pair $(a,b)$
  satisfies the following two conditions (which determines them uniquely):
  \begin{itemize}
  \item  $a+b=n-1$,
  \item $a=b$ if $n$ is odd,
    
  \item $a=b+1$ if $n$ is even.
  \end{itemize}
\begin{table}[h]
\scalebox{0.95}{
\begin{tabular}{||c|c|c|c||c|c||}
  \hline
$E$ & $e$ &$p$ &  $n$ & Structure & $\text{Generators}$\\
  \hline
--- & $1$ & $\neq 2$ & ---& $\F_q^\times \times \Z/p^{n-1} \times \Z/p^{n-1}$ & $\{\xi_{p^2-1},1+p,1+p\sqrt{d} \}$\\
\hline
--- & $1$ & $2$  & $\ge 2$ & $\F_4^\times \times \Z/2 \times \Z/2^{n-2} \times \Z/2^{n-1}$ & $\{ \xi_3,-1,5+4\sqrt{5}, \sqrt{5}\}$\\
  \hline
$\neq \Q_3(\sqrt{-3})$ & $2$ & $\neq 2$ & --- &$\F_p^\times \times \Z/p^a \times \Z/p^b$ & $\{\xi_{p-1},1+p,1+p\sqrt{d}\}$\\
  \hline
  $\Q_3(\sqrt{-3})$ & $2$ & $3$ & $\ge 2$ & $\F_3^\times \times \Z/3 \times \Z/3^{a-1} \times \Z/3^{b}$ & $\{-1,\xi_3,4,1+3\sqrt{-3} \}$\\
  \hline
  $\Q_2(\sqrt{-1})$ & $2$ & $2$  & $\ge 3$ & $\Z/4 \times \Z/2^{b-1} \times \Z/2^{a-1}$ & $\{\sqrt{-1},5,1+2\sqrt{-1}\}$\\
  \hline
  $\Q_2(\sqrt{3})$ & $2$ & $2$  & $\ge 5$ & $\Z/2 \times \Z/2^{a-1} \times \Z/2^{b}$ & $\{-1,\sqrt{3},1+2\sqrt{3} \}$\\
  \hline
    $\Q_2(\sqrt{2d})$ & $2$ & $2$ & $\ge 5$ & $\Z/2 \times \Z/2^{b-1} \times \Z/2^{a}
$ & $\{ -1,5,1+\sqrt{2d}\}$\\
  \hline
\end{tabular}
}
\caption{Group structure of $\On$.}
\label{table:groupstructure}
\end{table}
\label{thm:quotients}
\end{thm}

\begin{proof}
  See for example \cite{Ranum1910} or \cite[Chapter II]{Neukirch}. 
\end{proof}

\begin{remark}
  For completeness, the missing small values are:
  $(\Om/2)^\times \simeq \Z/2$ if $E/\Q_2$ is ramified; if
  $E=\Q_2(\sqrt{3})$ or $\Q_2(\sqrt{2d})$,
  $(\Om/\id{p}_2^3)^\times \simeq \Z/4$ and
  $(\Om/\id{p}_2^4)^\times \simeq \Z/4 \times \Z/2$.
\end{remark}

\begin{lemma}
  Let $E/\Q_p$ be a quadratic extension and let $\delta$ denote the
  valuation of the discriminant of $E$. The number of inertial type
  Galois orbits of primitive characters
  $\theta:E^\times \to \C^\times$ of conductor $n$ whose restriction
  to $\Q_p^\times$ matches the character of the extension $E/\Q_p$ is
  given in Table~\ref{table:characters}.
\begin{table}[h]
\begin{tabular}{||l|c|c|c||c||}
  \hline
  $E$ & $e$ & $p$ & $n$ & \# Prim. Char.\\
  \hline
  --- & $1$ & $\neq 2$ & --- &  $\sigma_0(p+1) $\\
  \hline
  --- & $1$ & $2$ & $\ge 3$ &  $4 $\\
  \hline
  --- & $2$ & $\neq 2$ & $1$ & $1$\\
  \hline
  $\neq \Q_3(\sqrt{-3})$ & $2$ & $\neq 2$ & $\stackrel{n \ge2}{odd \, \vert\, even}$ & $0 \, \vert \, 1$\\
  \hline
  $\delta =3$ & $2$ & $2$ & $\stackrel{n \ge 6}{odd \, \vert \, even}$ &  $0 \, \vert \, 1$\\
  \hline
    $\delta =3$ & $2$ & $2$ & $n=5$ &  $3$\\
  \hline
  $\Q_3(\sqrt{-3})$ & $2$ & $3$ & $2$ & $1$\\
  \hline
  $\Q_3(\sqrt{-3})$ & $2$ & $3$ & $\stackrel{n\ge 3}{odd \vert even}$ & $0 \, \vert \, 3$\\
  \hline
  $\delta=2$ & $2$ & $2$ & $3,4$  & $1$\\
  \hline
  $\delta=2$ & $2$ & $2$ & $\stackrel{n \ge 6}{odd \, \vert \, even}$ & $0 \, \vert \, 2$\\
\hline
\end{tabular}
\caption{Number of primitive characters}
\label{table:characters}
\end{table}
 \label{lemma:characters}
\end{lemma}

\begin{proof} Given the group structure and generators of
  Table~\ref{table:groupstructure}, it is enough to define a character
  in each of them.

  Suppose that $E/\Q_p$ is unramified and $p \neq 3$.  The condition
  $\theta|_{(\Z_p)^\times}=1$ implies that $\theta$ is trivial in the
  second generator. The primitive condition implies that its value at the
  third generator must be a primitive $p^{n-1}$ root of unity, and its
  value on $\xi_{p^2-1}$, is an element of order dividing $p+1$. Up to
  conjugation, the last value is the only free one, hence the total
  number equals $\sigma_0(p+1)$. The case $p=3$ works the same.  For
  $p=2$ there is $1$ for $n=1$, $2$ for $n=2$ and $4$ for $n\ge 3$.

  If $E/\Q_p$ is ramified, either $n=1$ (hence
  $\Om_E/\id{p} \simeq \Z/p$) in which case there is a unique
  character (namely that of $\epsilon_E$) or primitive characters only
  appear for even exponents. The reason is that for odd conductor
  exponents $(1+p)$ increases its order but $\theta$ is trivial on
  such element giving non-primitive characters. There are some
  exceptions, namely when $E/\Q_2$ is ramified. For example: if
  $\Disc(E/\Q_2)=2^2$, the condition
  $\theta|_{\Q_2^\times}=\epsilon_E$ implies that $n \ge3$ hence
  characters of conductor $3$ are primitive. If $E = \Q_2(\sqrt{2d})$
  then $\epsilon_E(5)=-1$ so $n \ge5$ and characters of conductor $5$
  are also primitive. The number of characters in each case follows
  easily from the generators and the group structure given in
  Table~\ref{table:groupstructure}.
\end{proof}

Supercuspidal automorphic representations correspond via local
Langlands to irreducible induced representations of a character
$\theta$ from a quadratic extension $E$, and the irreducibility
condition is equivalent to $\theta$ not factoring through the norm
map.

\begin{lemma}
  Let $E/\Q_p$ be a quadratic extension. The inertia type Galois
  orbits of characters $\theta$ that factor through the norm map
 are:
  \begin{enumerate}[i)]
  \item The trivial one (of conductor $0$).
  \item A conductor $1$ one if $E/\Q_p$ is unramified.
  \item A conductor $2$ and two of conductor $3$ if $E/\Q_2$ is unramified.
  \item A conductor $1$ one if $E/\Q_p$ is ramified and $p\equiv 1 \pmod 4$.
  \item Two quadratic of conductor $5$ for $E=\Q_2(\sqrt{2})$
    or $\Q_2(\sqrt{-6})$.
  \end{enumerate}
\label{lemma:normmap}
\end{lemma}
\begin{proof} Clearly the trivial character factors through the norm
  map.  Let $\epsilon_E$ be the quadratic character giving the
  extension $E/\Q_p$. Suppose that
  $\theta(\alpha) = \phi(\Norm(\alpha))$ for some character $\phi$ of
  $\Q_p^\times$. Since $\theta|_{\Q_p^\times}=\epsilon_E$,
  $\theta(a) = \phi(a^2)=\epsilon_E(a)$ for any $a \in
  \Q_p^\times$. In particular, $\epsilon_E$ is a square and if
  $p\neq 2$, $\cond(\phi)=\cond(\epsilon_E)$.
  \begin{itemize}
  \item If $E/\Q_p$ is unramified, the norm map is surjective, hence
    $\phi$ is uniquely determined by $\theta$ (and
    vice-versa). Since $\epsilon_E$ is trivial on $\Z_p^\times$, $\phi$
    is trivial on $(\Z_p^\times)^2$. If $p \neq 2$,
    $\Z_p^\times / (\Z_p^\times)^2$ is of order two, which give two
    possible characters $\phi$ namely the trivial one (with conductor
    $0$) and a ramified one of conductor $1$.
    
  \item If $E/\Q_2$ is unramified, $\Z_2^\times /(\Z_2^\times)^2$ has
    index $4$, we get one case of conductor $0$ (the trivial one), one
    case of conductor $2$ and two cases of conductor $3$.
  \item If $E/\Q_p$ is ramified and $p \neq 2$, the norm map is not
    surjective, being the image of $\Om_E^\times$ equal to
    $(\Z_p^\times)^2$. This determines $\phi$ uniquely, since if
    $\alpha \in \Om_E^\times$, there exists $a \in \Z_p^\times$ such
    that $\Norm(\alpha)=a^2$ hence
    $\theta(\alpha) = \phi(a^2)=\epsilon_E(a)$. In particular, $\phi$
    gives a square root of $\varepsilon_E|_{\Z_p^\times}$ so
    $p \equiv 1 \pmod 4$. Clearly there are two conjugate characters
    $\phi$ (of conductor $p$) whose square equals $\epsilon_E$.
  \item If $E/\Q_2$ is ramified, the condition $\epsilon_E(-1) = 1$
    implies that $\cond(\epsilon_E)=3$ and
    $\epsilon_E(3)=\epsilon_E(5)=-1$ (so $E=\Q_2(\sqrt{2})$ or
    $\Q_2(\sqrt{-6})$). The image of the norm map contains the squares
    with index $2$; since $\epsilon$ has order $2$, $\phi$ has order
    at most $4$, hence it factors through $(\Z_2/16)^\times$. Each
    field gives two possible quadratic characters $\phi$ as stated.
  \end{itemize}
\end{proof}

\begin{proof}[Proof of Theorem \ref{thm:mainthmodd}]
  The number of Galois conjugate local types of level $p^n$ is
  the following:

  \medskip
  
\noindent $\bullet$ {\bf Principal Series:} the local representation
is of the form $\pi(\chi_1,\chi_2)$. The Nebentypus being trivial
implies that $\chi_2|_{\Z_p^\times}=\chi_1^{-1}|_{\Z_p}^\times$ hence
$n=2\cond(\chi_1)$, i.e. they only contribute at even exponents.  Let
$d=n/2$. The restriction to inertia of $\chi_1$ is a primitive
character of $(\Z/p^d\Z)^\times$, a cyclic group of order
$(p-1)p^{d-1}$. The number of such characters (up to conjugation) is
precisely $\sigma_0(p-1)$ for $d >1$ and $\sigma_0(p-1)-1$ for $d=1$ (to
avoid the trivial character).

\medskip

\noindent $\bullet$ {\bf Special representations or Steinberg:} since the
Nebentypus is trivial, there are exactly two different types, of level $p$ and
$p^2$ respectively, with one type being the twist of the other by the quadratic
character ramified at $p$.

\medskip

\noindent $\bullet$ {\bf Supercuspidal Representations:} By
Theorem~\ref{thm:characterization} they are obtained by inducing a
character $\theta$, that does not factor through the norm map, from a
quadratic extension $E$ of $\Q_p$. As before, let $\epsilon_E$ denote
the character corresponding to the quadratic extension $E/\Q_p$.

\medskip

$\ast$ If $E/\Q_p$ is unramified (denoted S.C.U. in
Table~\ref{table:pneq3}), $n = 2\cond(\theta)$ and
$\theta|_{(\Z_p)^\times}=1$. By Lemma~\ref{lemma:characters} the total
number of such characters equals $\sigma_0(p+1)$, and by
Lemma~\ref{lemma:normmap} only two of them factor through the
norm map (the trivial one and a conductor $p$ one) for $n=2$.

\medskip $\ast$ If $E/\Q_p$ is ramified (denoted S.C.R. in
Table~\ref{table:pneq3}), $n = \cond(\theta)+\cond(\epsilon_E)$ and
$\theta|_{(\Z_p)^\times}=\epsilon_E|_{(\Z_p)^\times}$. If
$\cond(\theta) = 1$, there is a unique type by
Lemma~\ref{lemma:characters} and by Lemma \ref{lemma:normmap} the one
for $p \equiv 1 \pmod 4$ factors though the norm map. If
$p \equiv 3 \pmod 4$, the local type matches that of an unramified
supercuspidal representation (see for example \cite[Theorem
2.7]{Gerardin75}). Otherwise, Lemma~\ref{lemma:characters} implies
that primitive characters have even conductor (hence $n$ is odd) and 
there is a unique Galois inertial type orbit for each conductor except
when $p=3$ and $E=\Q_3(\sqrt{-3})$, in which case there are three.
\end{proof}

\subsection{The case $p=2$} This case is more delicate, and includes
the types corresponding to the \emph{sporadic supercuspidal
series}.

\subsubsection{Sporadic supercuspidal representations:}\label{sporadicones}
The projective image of the Weil group of $\Q_2$ might be one of the
sporadic cases $A_4$ or $S_4$. This phenomenon was studied by Weil in
\cite{Weil}, where he proved that the case $A_4$ does not occur over
$\Q$, while the case $S_4$ does. He also proved that there are
precisely $3$ extensions of $\Q_2$ with Galois group isomorphic to
$S_4$ and that there are precisely eight different cases with
projective image $S_4$ that correspond to the field extension of
$\Q_2$ obtained by adding the $3$-torsion points of the elliptic
curves
\begin{eqnarray}
    \label{eq:sporadicsupercuspidal1}
  E_1^{(r)}: ry^2 & =&  x^3+3x+2, \qquad r \in \{\pm1, \pm2\}, \\
\label{eq:sporadicsupercuspidal2}
  E_2^{(r)}: ry^2 & =&  x^3-3x+1, \qquad r \in \{\pm1, \pm2\}.  
\end{eqnarray}

A way to understand the problem is as follows: given an $S_4$
extension (equivalently, a representation $\rho:G \to S_4$, where
$G =\Gal(\overline{K}/K)$, for $K= \Q$ or $\Q_2$), compute all (if
any) representations $\tilde{\rho}$ of $K$ into $\GL_2(\C)$ whose
projectivization is isomorphic to $\rho$. Such general problem was
studied by Serre in \cite{Serreinv}. Let
$\tilde{S_4}\simeq \GL_2(\F_3)$ denote the quadratic extension of
$S_4$, where transpositions lift to involutions (see \cite{Serreinv}
page 654). Then two of the $S_4$ extensions lift to a representation of
$\tilde{S_4}$ while the other one does not (see Section $8$ of
\cite{Bayer}). The $2$-dimensional representations come from
(composing with) the faithful $2$-dimensional representation of
$\tilde{S_4}$. Note that the representations obtained from
\eqref{eq:sporadicsupercuspidal1} (respectively from
\eqref{eq:sporadicsupercuspidal2}) are quadratic twists of each other,
hence have the same projective image (and correspond to the two
extensions mentioned before).

Recall that two representations $\rho_i:G \to \GL_n(K)$, $i=1,2$ whose
projectivizations $\tilde{\rho_i}:G \to \PGL_n(K)$ are isomorphic are
twist of each other, i.e. there exists a character
$\chi : G \to K^\times$ such that $\rho_1 \simeq \rho_2 \otimes
\chi$. Since we only consider forms with trivial Nebentypus, all
sporadic supercuspidal representations are unramified twists of
(\ref{eq:sporadicsupercuspidal1}) and
(\ref{eq:sporadicsupercuspidal2}) so they cover all local types.
The level of such types is computed in \cite{Rio} (section 6). It
equals $2^7$ for the curves $E_1^{(r)}$, $2^4$ for $E_2^{(1)}$, $2^3$
for $E_2^{(-1)}$ and $2^6$ for $E_2^{(\pm1)}$.

\begin{thm}
The values of $\LT(2^n)$ are given in table~\ref{table:p=2}.
\begin{table}[h]
\begin{tabular}{||l|c|c|c|c|c|c||}
  \hline
$d$ & P.S. & Stb & S.C.U. & S.C.R(2) & S.C.R(3) & Sporadic\\
\hline
$1$ & --- & $1$ & --- & --- & --- & --- \\
\hline
$2$ & --- & --- & $1$ & --- & --- & --- \\
\hline
$3$ & --- &  --- & --- & --- & --- & $1$ \\
\hline
$4$ & $1$ & $1$ & $1$ & --- & --- & $1$ \\
\hline
$5$ & --- & --- & --- & $2$ & --- & --- \\
\hline
$6$ & $2$ & $2$ & $2$ & $2$ & --- & $2$ \\
\hline
$7$ & --- & --- & --- & --- & --- & $4$ \\
\hline
$8$ & $2$ & --- & $4$ & $4$ &--- & --- \\
\hline
$\stackrel{\ge 9}{\text{odd}}$ & --- & --- & --- & --- & $4$ & --- \\
\hline
$\stackrel{\ge 10}{\text{even}}$ & $2$ & --- & $4$ & $4$ & --- & --- \\
\hline
\end{tabular}
\caption{Types for $p=2$.}
\label{table:p=2}
\end{table}
\label{thm:mainthm2}
\end{thm}

\begin{proof}
  The strategy is the same as before, but more delicate.

\smallskip
  
\noindent $\bullet$ {\bf Principal Series:} this case mimics the odd
prime one with the difference that $(\Z/2^n)^\times$ is cyclic for $n=2$ but
isomorphic to $\Z/2 \times \Z/2^{n-2}$ if $n \ge 3$. Hence there is a
unique local type of conductor $4$, and two types for all other even
exponents. 

\smallskip

\noindent $\bullet$ {\bf Special representations or Steinberg:} there is a
unique automorphic form $\St$ of conductor $2$. There is one quadratic
character of conductor $2$ and two of conductor $3$ whose twists
give types of conductor $4$ and $6$ respectively.

  \smallskip

 \noindent $\bullet$ {\bf Supercuspidal Representations:} As in the odd case,
  we distinguish each possible extension $E/\Q_2$.

  \smallskip

  $\ast$ If $E/\Q_2$ is unramified (denoted S.C.U. in
  Table~\ref{table:p=2}), the level of the form equals $2\cond(\theta)$.
  There is one local type Galois orbit for
  $\cond(\theta)=1$ ($\theta$ being a cubic character), one for
  $\cond(\theta)=2$ (as the other one factors through the norm map),
  two for $\cond(\theta)=3$ (two factor through the norm map) and four for
  $\cond(\theta)>3$ (see Lemmas~\ref{lemma:characters} and \ref{lemma:normmap}).

\medskip
 
$\ast$ If $E/\Q_2$ is ramified with conductor $2$ (denoted S.C.R.(2)
in Table~\ref{table:p=2}), the level of the form equals
$2+ \cond(\theta)$. There are two such fields $E$, namely
$\Q_2(\sqrt{-1})$ and $\Q_2(\sqrt{3})$. By
Lemmas~\ref{lemma:characters} and \ref{lemma:normmap},
the number of such types equals:
\[
\begin{cases}
  0 & \text{ if }\cond(\theta)=1, 2 \text{ or }\cond(\theta)\ge 4 \text{ and odd},\\
  1 & \text{ if }\cond(\theta)=3,\\
  1 & \text{ if }\cond(\theta)= 4,\\
  2 & \text{ if }\cond(\theta)\ge 5 \text{ and even}.\\
\end{cases}
\]

\medskip

$\ast$ If $E/\Q_2$ is ramified with conductor $3$ (denoted S.C.R.(3)
in Table~\ref{table:p=2}), the level of the form equals
$3+ \cond(\theta)$. There are four such fields, namely
$\Q_2(\sqrt{2}), \Q_2(\sqrt{-2}), \Q_2(\sqrt{6})$ and
$\Q_2(\sqrt{-6})$. By Lemma~\ref{lemma:characters} the
number of Galois orbits equals:

\[
\begin{cases}
  0 & \text{ if }\cond(\theta)=2 \text{ or }\cond(\theta) \text{ is odd},\\
  3 & \text{ if }\cond(\theta)= 5,\\
  1 & \text{ if }\cond(\theta)\ge 5 \text{ and is even}.\\
\end{cases}
\]
Recall that for odd primes $p \equiv 3 \pmod 4$, a ramified type
matches that of an unramified one. When $p=2$, the same phenomenon
occurs in many cases. We refer to \cite[Section 41.3]{Bushnell06} for a
detailed description. Following their notation, all supercuspidal
representations are imprimitive (see Definition in page 255 and Lemma 41.3 of
\cite{Bushnell06}) and the way to test whether a local type appears
for different quadratic extensions is by computing the number of
quadratic twists that give isomorphic representations (denoted by
$I(\rho)$). In particular, if the form is triply imprimitive (i.e. it
comes from more than one quadratic extension), it must be the case
that $\frac{\theta}{\theta^\sigma}$ is a quadratic character, where
$\sigma$ generates $\Gal(E/\Q_2)$. With this criterion, the following
types are simply imprimitive:
\begin{itemize}
\item representations induced from $E/\Q_2$ ramified with discriminant
  valuation $2$ and $\cond(\theta) \ge 7$.
\item representations induced from $E/\Q_2$ ramified with discriminant
  valuation $3$ and $\cond(\theta) \ge 6$.
\end{itemize}

The case $E/\Q_2$ with discriminant $3$ and $\cond(\theta)=5$ is of
particular interest. For any $E$, $\epsilon_E(5)=-1$. Each field has
$3$ different local Galois orbits (by Lemma~\ref{lemma:characters})
two of order $2$ and one of order $4$; if $E=\Q_2(\sqrt{2})$ or
$\Q_2(\sqrt{-6})$, then by Lemma~\ref{lemma:normmap} two characters
factor through the norm map for each of them (when $\theta$ has order
$2$) which we discard.

Let $\theta$ be quadratic and let $\phi$ be any order $4$ character of
$(\Z/16)^ \times$. In particular, $\phi(9)=-1$, so
$\theta \cdot (\phi\circ \Norm)$ is trivial at $5$ hence gives a
character of conductor $3$ of $E$ (or the trivial character in the
discarded cases). In particular, the twist
$\BC \otimes \phi= \Ind_{W(E)}^{W(\Q_2)}(\theta \cdot (\phi\circ
\Norm))$
has conductor $3$ (with non-trivial Nebentypus) so by
\cite[Proposition 41.4]{Bushnell06} it matches the supercuspidal
unramified type, which was counted before.

If $\theta$ has order $4$, the representation
$\Ind_{W(E)}^{W(\Q_2)}\theta$ is triply imprimitive. An easy
computation proves that the set $I(\Ind_{W(E)}^{W(\Q_2)}\theta)$
equals:
\begin{itemize}
\item $\{1,\chi_{3},\chi_{2},\chi_6\}$ if $E=\Q_2(\sqrt{2})$,
\item $\{1,\chi_3,\chi_{-2},\chi_{-6}\}$ if $E=\Q_2(\sqrt{-6})$,
\item $\{1,\chi_3,\chi_{-2},\chi_{-6}\}$ if $E=\Q_2(\sqrt{-2})$,
\item $\{1,\chi_3,\chi_{2},\chi_{6}\}$ if $E=\Q_2(\sqrt{6})$,
\end{itemize}
where $\chi_i$ denotes the quadratic character of the extension
$\Q_2(\sqrt{i})$. In particular, all such local types match those from
$\Q_2(\sqrt{3})$, hence we do not need to count them again.

\medskip

\noindent $\bullet$ {\bf Sporadic supercuspidal representations:} the assertion follows from
Weil's results stated before.
\end{proof}

\section{Types from modular forms}
\label{sec:modformtypes}

Let $f = \sum_{n \ge 1} a_n q^n \in S_k(\Gamma_0(N))$ be a newform,
and let $\pi_f$ be the automorphic representation of $\GL_2(\A_\Q)$
attached to it. It is well known that $\pi_f$ is a restricted tensor
product $\bigotimes '_p \pi_{f,p} \otimes \pi_{f,\infty}$, where
$\pi_{f,p} \in \Aa_p$ is a representation of $\GL_2(\Q_p)$. Then for
each prime $p$, the form $f$ has attached a local type (that of
$\pi_{f,p}$). Let $K_f = \Q(a_n)$ denote the coefficient field of
$f$. If $N$ is a positive integer, let $\xi_N$ denote an $N$-th
primitive root of unity and $\Q(\xi_N)^+$ the maximal totally real
subextension of $\Q(\xi_N)$.

\begin{lemma}
  Let $f \in S_k(\Gamma_0(N))$ and let $p$ be a prime number. If
  $\pi_{f,p}$ is isomorphic to a Principal Series
  $\pi(\chi_1,\chi_2)$, where $\chi_1|_{\Z_p^\times}$ has order $d$,
  then $\Q(\xi_d)^+ \subset K_f$.
\label{lemma:coeffppalseries}
\end{lemma}

\begin{proof}
  Let $L = K_f \, \cap \, \Q(\xi_d)$. Suppose that
  $L \subsetneq \Q(\xi_d)^+$ and let $\ell \neq p$ be a prime such
  that there exists a prime ideal $\lambda$ of $\Om_L$ (the ring of
  integers of $L$) whose inertial degree in $\Q(\xi_d)^+$ is not
  $1$. Let
  $\rho_{f,\lambda}:\Gal(\overline{\Q}/\Q) \to \GL_2(K_{f,\lambda})$
  be the Galois representation attached to $f$ (by
  \cite{Deligne}). The restriction to the decomposition group at $p$
  matches (up to isomorphism) the representation
  $\chi_1 \oplus \chi_1^{-1} \chi_\ell^{k-1}$ (where $\chi_\ell$
  denotes the $\ell$-th cyclotomic character). Evaluating at elements
  of $\Z_p^\times$ (corresponding via Local-Langlands to elements in
  the inertia group) we see that $K_{f,\lambda}$ contains
  $\xi_d + \xi_d^{-1}$, which generates $\Q(\xi_d)^+$. But our
  assumption on $\lambda$ implies that the completion of $\Q(\xi_d)^+$
  (at a prime dividing $\lambda$) and $K_{f,\lambda}$ (at $\lambda$)
  are different, giving a contradiction.
\end{proof}

\begin{lemma}
  Let $f \in S_k(\Gamma_0(N))$ and let $p$ be a prime number. If
  $\pi_{f,p}$ is isomorphic to a Supercuspidal Representations, say
  $\pi_{f,p} = \Ind^{W(E)}_{W(\Q_p)}\theta$, where $\theta$ has order $d$, then
  $\Q(\xi_d)^+ \subset K_f$.
\label{lemma:coeffsupercuspidal}
\end{lemma}

\begin{proof}
  The restriction of $\rho_{f,\lambda}$ to $W(E)$ equals
  $\theta \oplus \theta'$, where if
  $\sigma \in W(\Q_p) \setminus W(E)$, then
  $\theta'(\mu) = \theta (\sigma \mu \sigma^{-1})$. The result
  follows from the same argument as the principal series case, via
  evaluating at elements of $\Z_p^\times$; note that the trivial Nebentypus
  condition implies that on such elements $\theta' = \theta^{-1}$.
\end{proof}

Lemmas~\ref{lemma:coeffppalseries} and \ref{lemma:coeffsupercuspidal}
imply that the coefficient field contains many automorphisms to
conjugate the form $f$. If we fix a prime $p$ dividing the level, the
global Galois orbit of the modular form $f$ contains representatives
for all elements of the same local type Galois orbit of $\pi_{f,p}$.

\begin{thm}
  Let $f \in S_k(\Gamma_0(N))$ be a newform and $p \mid N$ a prime
  number. Then the set
  $\{\pi_{\sigma(f),p} \; : \; \sigma \in \Gal(\C/\Q)\}$ of local
  types at $p$ of the Galois conjugates of $f$ equals the local type
  Galois orbit of $\pi_{f,p}$.
\label{thm:LTbound}
\end{thm}

\begin{proof} Clearly
  $\{\pi_{\sigma(f),p} \; : \; \sigma \in \Gal(\C/\Q)\}$ is contained
  in the local type Galois orbit of $\pi_{f,p}$, hence we need to
  prove the other containment.

  The result is clear when the local type of $\pi_{f,p}$ is Steinberg,
  as there is a unique element in the class.
  In the Principal Series case, note that $\pi(\chi_1,\chi_2)$ and
  $\pi(\chi_2,\chi_1)$ are isomorphic. Furthermore, the trivial
  Nebentypus hypothesis implies that
  $\chi_2|_{\Z_p^\times} = \chi_1^{-1}|_{\Z_p^\times}$. Suppose
  $\pi_{f,p}=\pi(\chi_1,\chi_2)$, where $\chi_1$ is a primitive
  character of order $d$ and conductor $n/2$. The Galois
  orbit of $\pi(\chi_1,\chi_2)$ has $\varphi(d)$ elements. Among such
  conjugates, $\varphi(d)/2$ are non-isomorphic when $d \neq 2$ and
  contains a unique element when
  $d=2$. Lemma~\ref{lemma:coeffppalseries} implies that $K_f$ contains
  $\Q(\xi_{d})^+$.
  \\

  \noindent{\bf Claim:} let $\sigma \in \Gal(\C/\Q)$. Then
  the local type of $\pi_{\sigma(f),p}$ is isomorphic to that of
  $\pi(\sigma(\chi_1),\sigma(\chi_2))$.

  \smallskip Let $\mu$ be the restriction of $\sigma$ to $\Q(\xi_d)^+$
  and let $\pi_{\mu(f),p}=\pi(\psi_1,\psi_2)$. The characters
  $(\psi_1,\psi_2)$ are determined by $\mu(f)$: the values
  $\{\psi_1(x),\psi_2(x)\}$ are roots of the polynomial
  $x^2-\mu(\chi_1+\chi_2)x + \chi_1 \chi_2 \in \Q(\xi_{d})^+$ (which
  is the characteristic polynomial of the image of $x$ under the
  Galois representations attached to $\mu(f)$). In particular they
  match the values $\{\sigma(\chi_1)(x),\sigma(\chi_2)(x)\}$. If
  $p\neq 2$, $(\Z/p^k)^\times$ is cyclic, so taking $x$ to be a
  generator, the characters are uniquely determined by their values on
  $x$. In particular, $\psi_1 = \sigma(\chi_1)$ or
  $\psi_1 = \sigma(\chi_2)$. For $p=2$,
  $(\Z/2^k)^\times = \Z/2 \times \Z/2^{k-2}$, and the trivial
  Nebentypus hypothesis imply that both characters $\psi_1$ and
  $\psi_2$ take the same value at the generator of the
  $\Z/2$-part. Then again, $\psi_1 = \sigma(\chi_1)$ or
  $\psi_1 = \sigma(\chi_2)$. Note that the two choices of
  $(\psi_1,\psi_2)$ are conjugate to each other, and give isomorphic
  local types (which explains the discrepancy between the action on
  characters of the group $\Gal(\Q(\xi_d)/\Q)$ and
  $\Gal(\Q(\xi_d)^+/\Q)$).

  \medskip
  
  The supercuspidal case follows from a similar computation using
  Lemma~\ref{lemma:coeffsupercuspidal} to get the different conjugates
  of the character $\theta$.
\end{proof}

While working with Galois orbits of modular forms, there is another
natural invariant to consider, namely the Atkin-Lehner eigenvalue
at each prime $p\mid N$. By the theory of Atkin and Lehner
(see\cite{Atkin1970}), if $f \in S_k(\Gamma_0(N))$ is a newform and
$p \mid N$, then $f$ is an eigenform for the A-L involution $W_p$, i.e.
$W_p(f) = \lambda_p f$, with $\lambda_p = \pm 1$.

\begin{lemma}
  Let $f \in S_k(\Gamma_0(N))$ and let $p \mid N$ a prime number such
  that $W_p(f) = \lambda_p f$. If $\sigma \in \Gal(\C/\Q)$ then
  $W_p(\sigma(f)) = \lambda_p \sigma(f)$.
\label{lemma:galoisinvariance}
\end{lemma}

\begin{proof}
  It is immediate from the fact that $W_p$ is an involution and
  commutes with the Hecke operators. In particular, the space
  $S_k(\Gamma_0(N),\Q) = S_k(\Gamma_0(N),\Q)^+ \oplus
  S_k(\Gamma_0(N),\Q)^-$, and the Hecke operators preserve both
  spaces.
\end{proof}

There is a delicate situation while computing A-L operators. If
$f \in S_k(\Gamma_0(N))$, it needs not be minimal among twists with
trivial Nebentypus. For example, if $f \in S_k(\Gamma_0(p))$, and we
look at forms in its Galois orbits $\{\sigma(f)\}$, we can twist them
by $\chi_p$ (the quadratic character unramified outside $p$) and get a
Galois orbit of new forms $\{\sigma(f) \otimes \chi_p\}$ in
$S_k(\Gamma_0(p^2))$. All such forms will have a predetermined A-L
sign, namely $\chi_p(-1)$ (see \cite[Theorem 6]{Atkin1970}), while the
A-L eigenvalue of $f$ at $p$ might take any value $\pm 1$, so we
``lost'' the invariant. Our final goal is to determine invariants of
Galois orbits of eigenforms, so we can either look at forms which have
minimal level up to (quadratic) twists, or to add the A-L sign of a
minimal twist.

\begin{dfn}
  Let $f \in S_k(\Gamma_0(N))$ be a newform, and let $p \mid
  N$. Consider the set of quadratic twists $T_f=\{f \otimes \psi\}$
  where $\psi$ ranges over all quadratic characters unramified outside $p$. Then
  either:
  \begin{enumerate}
  \item all elements in $T_f$ have level greater or equal than that of
    $f$ or
  \item there exists a unique form $g \in T_f$ of minimal level smaller than $N$.
  \end{enumerate}
  We define the \emph{minimal Atkin-Lehner sign} of $f$ at $p$ to be
  that of $f$ in the first case, and that of $g$ in the second one.
\label{defi:minimalAL}
\end{dfn}

The minimal A-L sign at $p$ of a newform $f$ is sometimes determined
by the local type $\tilde\pi_{f,p}$ of $f$ at $p$.

\begin{thm}
  \label{thm:AL}
  Let $p$ be a prime number and $\tau \in \Aa_p$ be such that
  $\tilde{\tau} = \tilde{\pi_{f,p}}$ for $f \in S_k(\Gamma_0(N))$ a newform. Then:
  \begin{enumerate}
  \item If $\tilde{\tau}$ is principal series, a ramified twist of Steinberg
    or a supercuspidal unramified representation (i.e. induced from an
    unramified quadratic extension of $\Q_p$), then the eigenvalue of
    the Atkin-Lehner involution $W_p$ is the same for all modular form
    $f$ with local type at $p$ $\tilde{\tau}$.
    
  \item If $\tilde{\tau}$ is Steinberg, or if $p \neq 2$ and
    $\tilde{\tau}$ is a ramified Supercuspidal representations induced from a
    character with even conductor, then there are two possible signs
    for the Atkin-Lehner involution for modular forms with local type
    at $p$ $\tilde{\tau}$. Furthermore, the two values are interchanged while
    twisting by the quadratic unramified character (which clearly
    preserves types).
    
  \item If $p=2$ and $\tilde{\tau}$ is a ramified supercuspidal representations
    induced from a character $\theta$ of a quadratic extension
    $E/\Q_2$ with discriminant valuation $2$, then:
    \begin{itemize}
    \item  there are two possible signs (interchanged by the
    unramified quadratic twist) when $\cond(\theta)=3$,
  \item there is a unique possible sign when $\cond(\theta)$ is even.
    \end{itemize}
If $E/\Q_2$ has discriminant valuation $3$, then
\begin{itemize}
\item  there is a unique sign for $\cond(\theta)=5$,
\item there are two possible signs for even conductors.
\end{itemize}

  \item If $p=2$ and $\tilde{\tau}$ is a sporadic supercuspidal representations, then
    the situation is as follows:
    \begin{itemize}
    \item If $\tilde{\tau}$ has level $2^7$ or $2^3$,
    then both Atkin-Lehner signs appear, and they are exchanged by the
    quadratic unramified twist.
  \item If $\tilde{\tau}$ has level $2^4$ or $2^6$, then the quadratic
    unramified twist preserves the local sign, but these types are
    not-minimal, they are twists of the level $2^3$ one.
    \end{itemize}
  \end{enumerate}
\end{thm}

\begin{proof}
  The result is well known to experts, and follows from the
  characterization of the local sign of automorphic forms given by
  Deligne (see \cite{Deligne1973} and also \cite{Schmidt}). In the $1$-dimensional case it is
  clear that the local root number is determined by the restriction to
  inertia of the character as well as its value at a local uniformizer (see
  for example (3.4.3.2) of \cite{Deligne1973}).
  \begin{itemize}
  \item Suppose that $\tilde{\tau}$ is principal series, and
    $\tilde{\pi}(\chi_1,\chi_2) \in \tilde{\tau}$. Then the local sign of
    $\pi(\chi_1,\chi_2)$ equals the product of the two local
    signs. But the trivial Nebentypus hypothesis implies that the
    product of the two characters evaluated at $p$ is uniquely
    determined, hence their restriction to inertia determines uniquely
    the sign of $\pi(\chi_1,\chi_2)$.
    
  \item The Steinberg case is well understood. In this case the
    Atkin-Lehner involution at $p$ is related to the $p$-th Fourier
    coefficient $\lambda_p(f)p^{\frac{k-2}{2}} = -a_p(f)$. Note that
    the Weil-representation equals
    $\omega^{\frac{k}{2}-1}(\psi \oplus \psi\omega)$, where $\omega$
    is the unramified quasi-character giving the action of $W(\Q_p)$
    on the roots of unity and $\psi$ is a quadratic unramified
    character. Then $\lambda_p(f) = -\psi(p)$. Clearly twisting by the
    character corresponding to the unramified quadratic extension of
    $\Q_p$ changes the A-L eigenvalue.

    The ramified twist of Steinberg case is well known (see for
    example \cite[Theorem 6]{Atkin1970}). It can also be recovered by
    studying the local sign variation under twisting (see (3.4.3.5)
    and Theorem 4.1 (1) of \cite{Deligne1973})

  \item If $\tilde{\tau}$ is supercuspidal representations, the local factor can
    be explicitly computed (following \cite{Deligne1973}). Recall that
    one of the local sign properties (see \cite[3.12
    (C)]{Deligne1973}) is
  \[
\varepsilon\left(\BC,\psi,dx\right) = \varepsilon\left(\theta,\psi \circ \Trace,dx\right).
\]
The Swan conductor of $\theta$, denoted $sw(\theta)$, equals $0$ if
$\theta$ is unramified and $\cond(\theta)-1$ otherwise. Let
$s=\cond(\psi \circ \Trace)+sw(\theta)+1$ and let $\pi$ be a local
uniformizer. By  \cite[page 528]{Deligne1973},
    \begin{equation}
      \label{eq:localsign}
      \varepsilon\left(\BC,\psi,dx\right) = \theta(\pi)^{s} \int_{\Om^\times} \theta^{-1}(x)\psi\circ \Trace\left(\frac{x}{\pi^s}\right)d\frac{x}{\pi^s}.
    \end{equation}
    In particular, the local sign depends on the restriction of
    $\theta$ to $\Om^\times$ and its value in a local
    uniformizer. Recall that the determinant of the representation
    equals $\epsilon_E \theta|_{\Q_p^\times}$, hence the value of 
    $\theta(p)$ is uniquely determined.  If $E/\Q_p$ is unramified,
    $p$ is a local uniformizer, hence the local sign only depends on
    the Weil-Deligne type.

    If $E/\Q_p$ is ramified and $\pi$ is a local uniformizer, the
    trivial Nebentypus condition determines the value of 
    $\theta(p) = \theta(\pi^2)$, but not that of $\theta(\pi)$. Chose
    $\psi$ to be an additive character with conductor $0$ (i.e. it is
    trivial on $\Z_p$ but non-trivial on $\frac{1}{p}\Z_p$). Then
    clearly $\cond(\psi \circ \Trace)\equiv v_p(\Disc(E)) \pmod 2$.
    \begin{itemize}
    \item If $p \neq 2$, $v_p(\Disc(E)) \equiv 1 \pmod 2$ hence if
      $\cond(\theta)=1$, $s$ is even and the local sign is uniquely
      determined (recall that such type matches the unramified
      one!). If $\cond(\theta)$ is even, $s$ is odd (by
      Lemma~\ref{lemma:characters}) hence there are two possible
      signs. Furthermore, we can move from one sign to the other
      twisting by the unramified quadratic character (which changes
      the sign of $\theta(\pi)$).
    \item If $p=2$ and $v_2(\Disc(E/\Q_2))=2$,
      $s \equiv \cond(\theta) \pmod 2$, hence the sign is uniquely
      determined for all $\theta$ of even conductor. If
      $\cond(\theta)=3$, there are two possibilities (corresponding to
      modular forms of level $2^5$). The forms of level $2^6$ are
      quadratic twists of these ones, hence although the local sign is
      uniquely determined, they have two possible minimal Atkin-Lehner
      signs at $2$.

      At last, if $v_2(\Disc(E/\Q_2)=3$,
      $s \equiv \cond(\theta)+1 \pmod 2$, so the sign is uniquely
      determined for $\cond(\theta)=5$ (recall that this case matches
      the unramified one) while there are two possible values for even
      conductors (and both signs change by a local twist).
  \end{itemize}
\item Suppose $p=2$ and $\tau$ is sporadic supercuspidal representations, so
  the Weil representation $\rho$ attached to it has image isomorphic
  to $\tilde{S_4}$, i.e. there exists $E/\Q_2$ with
  $\Gal(E/\Q_2) \simeq \tilde{S_4}$.

  The character table of $\tilde{S_4} \simeq \GL_2(\F_3)$ is recalled
  in Table~\ref{table:chartabletS4}. The representations $\text{Sg}$,
  $\text{St}_2$ and $\text{St}_3$ are the representations obtained
  from quotients of $\PGL_2(\F_3) \simeq S_4$, and they are the sign
  representation, the $2$-dimensional standard representation obtained
  from the isomorphism
  $S_4 /\langle (12)(34),(13)(24)\rangle \simeq S_3$, and the
  $3$-dimensional representation of $S_4$. The representation $V$ is
  the alluded in Section~\ref{sporadicones}. 

  Another description of such representations come from the group
  $\GL_2(\F_3)$: the two $1$-dimensional ones are the ones factoring
  through the determinant. The last $3$ representations come from
  ``principal series'': if $\chi$ is the non-trivial character of
  $\F_3^\times$, $\pi(\chi,1)$ gives the irreducible four dimensional
  representation; $\pi(1,1)$ and $\pi(\chi,\chi)$ have an irreducible
  quotient/subspace of dimension three (the ``Steinberg''
  ones). Finally, the two dimensional ones, can be constructed as
  follows: identify $\F_9^\times$ with the non-split Cartan
  ${\mathcal C}_{ns}=\left\{\left(\begin{smallmatrix} a & -b\\b & a\end{smallmatrix}
    \right) \in \GL_2(\F_3)\right\}$;
  pick $\psi$ a non-trivial additive character of $\F_3$ and let
  $\theta : \F_9^\times \to \C^\times$ be a character. Let
  $\theta_\psi$ be the character in
  $M=\left\{Z(\GL_2(\F_3)) \cdot \left(\begin{smallmatrix} 1 & \F_3\\
        0 & 1\end{smallmatrix} \right)\right\}$
  given by
  $\theta_\psi\left(\left(\begin{smallmatrix} a & 0\\0 &
        a\end{smallmatrix}\right) \left(\begin{smallmatrix} 1 & u \\ 0
        & 1\end{smallmatrix} \right) \right) = \theta(a) \psi(u)$.
  Then if $\theta$ is not trivial nor quadratic, the virtual
  representation
  $\Ind_M^{\GL_2(\F_3)} \theta_\psi - \Ind_{{\mathcal C}_{ns}}^{\GL_2(\F_3)} \theta$ is
  an irreducible representation independent of $\psi$ (see
  \cite[Theorem 6.4]{Bushnell06}). If $\theta$ has order $8$, we get
  the representation $V$ and its twist, while $\theta$ of order $4$
  gives the representation $\text{St}_2$.

  \begin{table}[h]
\begin{tabular}{||l|c|c|c|c|c|c|c|c||}
  \hline
  & $\left(\begin{smallmatrix}1&0\\0&1\end{smallmatrix} \right)$ & $\left(\begin{smallmatrix}2&0\\0&2\end{smallmatrix} \right)$ & $\left(\begin{smallmatrix}0&1\\2&0\end{smallmatrix} \right)$ & $\left(\begin{smallmatrix}0&1\\1&2\end{smallmatrix} \right)$ & $\left(\begin{smallmatrix}0&1\\1&1\end{smallmatrix} \right)$ & $\left(\begin{smallmatrix}1&1\\0&1\end{smallmatrix} \right)$ & $\left(\begin{smallmatrix}2&1\\0&2\end{smallmatrix} \right)$ & $\left(\begin{smallmatrix}1&0\\0&2\end{smallmatrix} \right)$\\
  \hline
  $\mathbbm{1}$ & $1$ & $1$ & $1$ & $1$ & $1$ & $1$ & $1$ & $1$\\
  \hline
  Sg & $1$ & $1$ & $1$ & $-1$ & $-1$ & $1$ & $1$ & $-1$\\
  \hline
  $\text{St}_2$ & $2$ & $2$ & $2$ & $0$ & $0$ & $-1$ & $-1$ & $0$\\
\hline
$V$ & $2$ & $-2$ & $0$ & $\sqrt{-2}$ & $-\sqrt{-2}$ & $-1$ & $1$ & $0$\\
  \hline
$V \otimes \text{Sg}$ & $2$ & $-2$ & $0$ & $-\sqrt{-2}$ & $\sqrt{-2}$ & $-1$ & $1$ & $0$\\
\hline
  $\text{St}_3$ & $3$  & $3$ & $-1$ & $-1$ & $-1$ & $0$ & $0$ & $1$\\
  \hline
  $\text{St}_3 \otimes \text{Sg}$ & $3$  & $3$ & $-1$ & $1$ & $1$ & $0$ & $0$ & $-1$\\
  \hline
$W$ & $4$ & $-4$ & $0$ & $0$ & $0$ & $1$ & $-1$ & $0$\\
\hline
\end{tabular}
\caption{Character table for $\GL_2(\F_3)$.}
\label{table:chartabletS4}
\end{table}
Consider the following subgroups of $\GL_2(\F_3)$:
$C_4 = \langle \left(\begin{smallmatrix} 0 & 1\\ 2 &
    0\end{smallmatrix} \right)\rangle$,
$C_6 = \langle\left(\begin{smallmatrix}2 & 1\\0 & 2\end{smallmatrix}
\right)\rangle$ and
$C_8 = \langle\left(\begin{smallmatrix}0 & 1\\1 & 2\end{smallmatrix}
\right)\rangle$. Using the character table and Frobenius reciprocity, it is
  easy to verify the following formulas
  \begin{eqnarray}
    \label{eq:inducedrepresentations}
    \Ind_{C_4}^{\GL_2(\F_3)}\chi_4 &\simeq& V \oplus (V \otimes \text{Sg}) \oplus 2W\\
    \Ind_{C_6}^{\GL_2(\F_3)}\chi_6 &\simeq& V \oplus (V \otimes \text{Sg}) \oplus W\\
    \Ind_{C_8}^{\GL_2(\F_3)}\chi_8 &\simeq& (V\otimes \text{Sg}) \oplus W,
  \end{eqnarray}
  where $\chi_j$ is a character of order $j$ in the corresponding
  group and we chose
  $\chi_8\left(\left(\begin{smallmatrix}0& 1\\1&
        2\end{smallmatrix}\right) \right) = \exp\left(-\frac{\pi
      i}{4}\right)$. Then
\begin{equation}
  \label{eq:maincharidentity}
V \simeq \Ind_{C_8}^{\GL_2(\F_3)}\chi_8 - \Ind_{C_4}^{\GL_2(\F_3)}\chi_4 + \Ind_{C_6}^{\GL_2(\F_3)} \chi_6.  
\end{equation}
To compute the sign variation, we can consider the formal
representation $\kappa V - V$, where $\kappa$ is the quadratic
unramified character of $\Q_2$. Using (\ref{eq:maincharidentity})
and the local sign formalism (\cite[Theorem 4.1]{Deligne1973}) we obtain
\begin{equation}
    \label{eq:rootnumbersporadic}
  \begin{split}
    \varepsilon(\kappa V - V,\psi,dx) = 
  \frac{\varepsilon(\kappa\chi_8,\psi \circ \Trace_{K_{C_8}},dx)}{\varepsilon(\chi_8,\psi \circ \Trace_{K_{C_8}},dx)}   \frac{\varepsilon(\chi_4,\psi \circ \Trace_{K_{C_4}},dx)}{\varepsilon(\kappa\chi_4,\psi \circ \Trace_{K_{C_4}},dx)}\\
  \frac{\varepsilon(\kappa\chi_6,\psi \circ \Trace_{K_{C_6}},dx)}{\varepsilon(\chi_6,\psi \circ \Trace_{K_{C_6}},dx)}
\end{split}
\end{equation}
The characters in (\ref{eq:rootnumbersporadic}) are understood as
class field characters, giving the corresponding field extension.
Recall that each sign variation depends only on the value
$\kappa(\Norm(\pi_2))^s$, where $\pi_2$ is a local uniformizer and
$s = \val_2(\Disc(K_i))+ sw(\chi_i) + 1$. In particular, we need to
compute the inertial degree of $K_i$ and $s$ for each field. The
extensions $K_{C_4}$, $K_{C_6}$ and $K_{C_8}$ are contained in the
fixed field of
$\left(\begin{smallmatrix}-1& 0\\0& -1\end{smallmatrix}\right)$, a
Galois extension with Galois group isomorphic to $S_4$. The
ramification indices are $f(K_{C_4}/\Q_2)=2$, $f(K_{C_6}/\Q_2)=2$ and
$f(K_{C_8}/\Q_2)=1$. Then the sign contribution is trivial for the
first two ones (as $\Norm(\pi_2)$ is a square), and only depends on
the first term of (\ref{eq:rootnumbersporadic}). Furthermore,
$2 \mid \val_2(\Disc(K_{C_8}))$, hence
$s \equiv sw(\chi_8)+1 \pmod 2$.

To compute $sw(\chi_8)$, we consider the field extensions
$\Q_2 \subset K_{C_8} \subset K_{C_4} \subset K_{-1} \subset E$. The
group $C_8$ has characters of orders: $1$, $2$ (both unramified),
$4$ and $8$. The conductor discriminant formula gives the equality
\[
\Disc(E/K_{C_8}) = \prod_{\theta} \cond(\theta).
\]
Let $\theta_i$ denote the corresponding character of order $i$ (so
$\theta_8 = \chi_8$). The relative discriminant formula provides the
equations:
\begin{eqnarray}
  \label{eq:relativedisc}
  \Disc(E/K_{C_8}) & = & \cond(\theta_8)^2\cond(\theta_4)^2.\\
\Disc(K_{-1}/K_{C_8}) & = & \cond(\theta_4)^2.\\
\Disc(K_{-1}/\Q_2) & = & \Norm(\Disc(K_{-1}/K_{C_8})) \Disc(K_{C_8}/\Q_2)^4.\\
\Disc(E/\Q_2) & = & \Norm(\Disc(E/K_{C_8})) \Disc(K_{C_8}/\Q_2)^8.
\end{eqnarray}
Then computing for each of the $8$ fields the values $\Disc(E/\Q_2)$,
$\Disc(K_{-1}/\Q_2)$ and $\Disc(K_{C_8}/\Q_2)$, a simple manipulation determines
$sw(\chi_8)$.

Equations for the $8$ extensions appear in the online tables of
\cite{Roberts}. Note that in $\GL_2(\F_3)$ there are two non-conjugate
subgroups of order $8$, hence each extension can be obtained by two
different degree $8$ polynomials. The extensions are obtained as the
Galois closure of the polynomials:
    \begin{itemize}
    \item $x^8+20x^2+20$, $x^8+28x^2+20$, $x^8+6x^6+20$ and $x^8+2x^6+20$.
    \item $x^8+4x^7+4x^2+14$, $x^8+4x^7+12x^2+2$, $x^8+4x^7+12x^2+14$ and $x^8+4x^7+12x^2+10$.
    \end{itemize}
    The values of $\Disc(E/\Q_2)$ (for each extension) already appeared in
    \cite[Table 10]{Rio}), and equal: $64$, $76$, $100$ and $100$ for
    the first four fields and $136$ for all fields in the second list.
The other discriminants as well as the value of $\sw(\chi_8)$ are given in Table \ref{table:discriminants}, which proves the stated result.
\begin{table}[h]
\scalebox{0.8}{
\begin{tabular}{||c|c|c|c|c||}
\hline
Polynomial & $\val_2(\Disc(E/\Q_2))$ & $\val_2(\Disc(K_{-1}/\Q_2))$ & $\val_2(\Disc(K_{C_8}/\Q_2))$ & $\val_2(\cond(\chi_8))$\\
\hline
$x^8+20x^2+20$&$64$ & $28$ & $6$ & $1$\\
\hline
$x^8+28x^2+20$&$76$ & $28$ & $6$ & $18$\\
\hline
$x^8+6x^6+20$&$100$ & $28$ & $6$ & $2$\\
\hline
$x^8+2x^6+20$&$100$ & $28$ & $6$ & $2$\\
\hline
$x^8+4x^7+4x^2+14$&$136$ &$52$ & $10$ & $11$\\
\hline
$x^8+4x^7+12x^2+2$&$136$ &$52$ & $10$ & $11$\\
\hline
$x^8+4x^7+12x^2+14$&$136$ &$52$ & $10$ & $11$\\
\hline
$x^8+4x^7+12x^2+10$&$136$ &$52$ & $10$ & $11$\\
\hline
\end{tabular}}
\caption{Discriminant and conductor table.}
\label{table:discriminants}
\end{table}
\end{itemize}
\end{proof}
\begin{remark}
  The Atkin-Lehner sign of sporadic supercuspidal representations of level
  $2^4$ and $2^6$ should be the same for all of them (and equal to
  $+1$). The proof should follow the Steinberg case proved before
  using the trivial Nebentypus hypothesis.
\end{remark}
\begin{remark}
  The local Atkin-Lehner sign statement in \cite[Remark
  11]{Pacetti2010} is not correct. In the case of a supercuspidal
  representation, the correct local computation is the one done in the
  previous proof.
\end{remark}
To a newform $f \in S_k(\Gamma_0(N))$ and $p\mid N$ a prime, we can
attach the pair $(\tilde\pi_{f,p},\lambda_p)$ consisting of the local
type of $f$ at $p$, and its minimal Atkin-Lehner sign.

\begin{dfn}
  Let $\LO(p^n)$ denote the number of pairs $(\tilde{\tau},\epsilon)$
  where $\tilde\tau$ is a local type Galois orbit of level $p^n$ and
  $\epsilon$ is a compatible minimal Atkin-Lehner eigenvalue (i.e. the
  existence of a newform $f \in S_k(\Gamma_0(p^n))$ for some $k$ with
  $(\tilde\pi_{f,p},\lambda_p) = (\tilde\tau,\epsilon)$ does not
  contradict Theorem~\ref{thm:AL}).
\end{dfn}

\begin{thm}
The values of $\LO(p^n)$ are given in Table~\ref{table:valuesLO(p^n)}.
\begin{table}[h]
\begin{tabular}{||l|c|c|c||}
  \hline
$n$ & $\gcd(p,6)=1$ & $p=3$ & $p=2$\\
\hline
  $0$ & $1$ & $1$ & $1$\\
\hline
$1$ &$2$& $2$ & $2$\\
\hline
$2$ &$\sigma_0(p+1)+\sigma_0(p-1)-1$  & $9$ &$1$\\
\hline
$3$ &$4$& $8$&$2$\\
\hline
$4$ &$\sigma_0(p+1)+\sigma_0(p-1)$ & $10$  &$6$\\
\hline
$5$ &$4$ & $8$&$4$\\
\hline
$6$ &$\sigma_0(p+1)+\sigma_0(p-1)$ & $10$  &$16$\\
\hline
$\stackrel{\ge 7}{\text{odd}}$ &$4$  & $8$&$8$ \\
\hline
$\stackrel{\ge 8}{\text{even}}$ &$\sigma_0(p+1)+\sigma_0(p-1)$ & $10$  &$10$\\
\hline
\end{tabular}
\caption{The values of $\LO(p^n)$.}
\label{table:valuesLO(p^n)}
\end{table}
\label{thm:valuesofLO}
\end{thm}
\begin{proof}
  The result comes from Theorem~\ref{thm:mainthmodd},
  Theorem~\ref{thm:mainthm2} and Theorem~\ref{thm:AL}.
\end{proof}

\section{Existence of local types with compatible Atkin-Lehner sign}
\label{section:existence}

\begin{thm} Let $N$ be a positive integer such that $N$ is a prime
  power or $N$ is square-free. For each prime $q \mid N$, let
  $\tilde\tau_q$ be a local type of level $q^{\val_q(N)}$ and let
  $\epsilon_q \in \{\pm 1\}$ be a compatible Atkin-Lehner sign for
  $\tilde\tau_q$. Then there exists a positive integer $k_0$ such that
  for any $k \ge k_0$, there exists a newform $f \in S_k(\Gamma_0(N)$
  such that:
  \begin{enumerate}
  \item $\tilde\pi_{f,q} \simeq \tilde{\tau_q}$ for all primes $q$,
  \item the Atkin-Lehner eigenvalue of $f$ at $q$ equals $\epsilon_q$,
    
  \item $f$ does not have complex multiplication.
\end{enumerate}
\label{thm:typesexistence}
\end{thm}
\begin{proof}
  A very similar result in this direction is Theorem 1.1 of
  \cite{Weinstein2009} (see also Theorem 4.3), where an asymptotic
  formula for the number of types in the space of cusp forms of level
  $N$ is given for $k$ large enough. An important feature of its proof
  is that such number grows linearly in the weight $k$ (for $k$ big
  enough). Unfortunately, the result only counts types, not the whole
  local representation (so we do not get any information on the
  Atkin-Lehner signs); still, in the cases where the is a unique
  Atkin-Lehner sign at each local type, for example the case of
  modular forms whose local types are all principal series (see
  Theorem~\ref{thm:AL}) Weinstein result is indeed enough for our purposes.

  In the work \cite{Kimball} (Theorem 3.3) the existence of forms with
  any combination of local Atkin-Lehner signs is proven for $N$
  square-free (i.e. only Steinberg local types). A different approach
  is given in \cite{gross} (Section 10), were using the trace formula,
  the existence of automorphic forms for the group
  $\PGL_2$ with any supercuspidal local representations at a finite
  set of primes (of $\PGL_2(\Q_p)$) is proven. Gross' result is generalized in
  \cite{Templier}. Using the trace formula ideas (as in Gross'
  article), they prove (Theorem 1.2) that if $G$ is any connected
  reductive group over a totally real field, the number of automorphic
  forms of weight $k$ and level $N$ with prescribed local
  representations (which are supercuspidal at ramified primes) grows
  linearly with $k$ (recall that $\dim(\xi)=k-1$ if $\xi$ is the
  discrete series of weight $k$, which gives the linear
  growth). Furthermore, the result can be extended to include
  Steinberg types as done in Section 6 loc. cit. where a similar result
  is proven in Theorem 6.4. 

  Therefore the results in the aforementioned articles prove the first
  two claims of the theorem, and in most situations this is also
  enough to get the last one (as complex multiplication forms are
  supercuspidal at all primes). In the general setting, the number of
  complex multiplication forms with a fixed level $N$ is bounded as a
  function on the weight (see for example \cite[Corollary
  4.5]{Tsaknias2012a}). On the other hand, the existence results
  stated above imply that the forms satisfying the first two
  conditions grows linearly in $k$, hence for $k$ large enough the
  space always contains a non-CM modular form (of any given local inertial type).
\end{proof}

\begin{remark}
  The constant $k_0$ in the last theorem can be made explicit by
  computing all the constants involved in the cited articles; we did
  not pursue this objective. We expect the previous result to hold in
  general, but we did not find a suitable reference for it. Note that
  the proof given looks stronger than the Theorem itself, as it
  involves a control on the whole local representation. Such control
  does not hold in general, namely we cannot fix a principal series
  representation and expect it to appear in a modular form. The reason
  is that fixing a principal series ``involves'' fixing the value of
  the $p$-th eigenvalue as well (if the representation is unramified,
  it implies fixing the Hecke eigenvalue, while in the ramified case
  it implies fixing the Hecke eigenvalue of a base change of the
  form), which is a very strong condition. However, once we know that
  local types do exist (by Weinstein result) we are only asking for
  unramified twists of a type that appears in the space of modular
  forms to appear as well. This weaker statement should be easier to
  prove, but we do not have a direct proof of it.
\label{remark:localtypegeneralN}
\end{remark}

\begin{thm}\label{cor:NCMlowbound}
  Let $N$ be a prime power or square free. Then there
  exists $k_0$ such that for $k \ge k_0$,
  \begin{equation}
    \label{eq:Maeda}
 \prod_{p \mid N} \LO(p^{v_p(N)}) \le \NCM(N,k).    
  \end{equation}
\end{thm}

\begin{proof}
  By Theorem \ref{thm:typesexistence} we know that there exists $k_0$
  such that for $k \ge k_0$ and for each local type with a compatible
  A-L sign, a modular form $f$ of weight $k$ and level $N$ exists with
  the specified local type and Atkin-Lehner
  eigenvalue. Theorem~\ref{thm:LTbound} implies that Galois conjugate
  local types appear in the same Galois orbit of $f$, which gives the
  desired inequality.
\end{proof}

\begin{remark}
  If Theorem~\ref{thm:typesexistence} holds in general as explained in
  Remark~\ref{remark:localtypegeneralN}, then for any positive integer
  $N$ we get the inequality
  \begin{equation}
    \label{eq:generalN}
\prod_{p \mid N} \LO(p^{\val_p(N)})\le \NCM(N,k),        
  \end{equation}
for $k$ large enough.
\end{remark}
A natural question is to study how sharp is the inequality in
(\ref{eq:generalN}) for general $N$.  It is not true that the first
inequality is an equality in general!  The reason is that when $N$ is
a prime power (or a prime power times a square-free integer), there
are enough automorphisms in the coefficient field to conjugate each of
the local types so as to get the whole local Galois orbit for each of
them. The problem arises when the automorphisms needed for two
different primes correspond to the same extension (see
Lemmas~\ref{lemma:coeffppalseries} and
\ref{lemma:coeffsupercuspidal}). Here is a concrete example: suppose
that $N=11^2 \cdot 31^2$. Let $\tau_{11}$ be a principal series
corresponding to an order $5$ character, and $\tau_{31}$ be a
principal series of an order $5$ character as well. Let
$f \in S_k(\Gamma_0(11^2 \cdot 31^2))$ be a newform with the chosen
local types at $11$ and $31$. Lemma~\ref{lemma:coeffppalseries}
implies that $\Q(\xi_5)^+$ is contained in the coefficient field of
$f$, so conjugating we can fix a local type at $11$ in the orbit. Once
we fixed such type at $11$, we cannot conjugate the type at $31$
(globally), so we get $2$ different types at $31$ in the Galois orbit
of $f$. In this case, using Theorem~\ref{thm:typesexistence} we get
$2$ as a lower bound for $\NCM(11^2 \cdot 31^2,k)$ (for $k$ large
enough) instead of $1$.

With this example in mind, and the techniques developed before, one
can give a better but more involved lower bound formula for the number
of Galois orbits of modular forms of general level $N$ and large
enough weight $k$, assuming that Theorem~\ref{thm:typesexistence}
holds in general. However, in many instances (for example when $N$ has
a unique prime whose square divides it, or if it holds that whenever
$p^r \mid N$ and $q^s \mid N$, $\gcd((p-1)p,(q-1)q) = 1$) the product
of local Galois types is the best possible bound with our method. This
is precisely the case for the data gathered in
\cite{Tsaknias2012a}.

\medskip

Another natural question is the existence of other Galois orbit
invariants. Based on numerical computations done by the third author
(see \cite{Tsaknias2012a}) it seems that the answer should be
negative, hence we propose the following problem.

\begin{question}
  If $N$ is a prime power or square-free, is (\ref{eq:Maeda}) an
  equality? I.e. is it true that that for $k$ large enough the number
  of Galois orbits of modular forms of level $N$ equals the number of
  Galois conjugate local types with compatible Atkin-Lehner signs?
\label{conecture:maeda}
\end{question}

\begin{remark}
\label{remark:maeda}
  Due to the existence result (Theorem~\ref{remark:localtypegeneralN})
  an affirmative answer to Question~\ref{conecture:maeda} is
  equivalent to a uniqueness result (for $k$ large enough) for the
  Galois orbits of newforms with given Galois conjugate local type and
  compatible Atkin-Lehner signs.
\end{remark}

Clearly such statement is in the spirit of Maeda's original
conjecture, hence it seems natural to expect that if there is no
reason for forms to be non-conjugate, then they should be
conjugate. We do not claim that we are convinced on the veracity of
the problem, but we want to stress that numerical experiments suggest
that the answer might be positive (see \cite{Tsaknias2012a}) as the
values of $\LO(p^n)$ seem to match the number of orbits of non-CM newforms in the
respective space of modular forms of weight $k$ starting at very small
values of $k$.

However, for $p=2$ and $n \ge 8$ even, there is a discrepancy
that we cannot explain. 

\begin{example}
\label{example:discrepancy}
  Let $N=2^8$ and $k=12$. The space $S_{12}(\Gamma_0(256))$ contains
  $17$ Galois orbits. Five of them corresponds to CM forms (four with
  rational coefficients, and one whose coefficient field is
  quadratic). The remaining $12$ orbits have dimensions:
  $2, 2, 4, 4, 6, 6, 8, 8, 8, 10, 10, 12$. Computing a few Hecke
  operators, it can be checked:
  \begin{itemize}
  \item The $2$-dimensional ones are twists of each other
  (via $\chi_{-1}$) and that each orbit is stable under twisting by
  $\chi_{-2}$. It corresponds to the unramified supercuspidal
  representations. 
\item The $4$-dimensional orbits are stable under twisting by
  $\chi_{-1}$ hence are induced from $E=\Q_2(\sqrt{-1})$ and are twist
  of each other by $\chi_2$.
\item The same is true for the $6$-dimensional ones. 
\item Two of the $8$-dimensional ones have Galois orbits invariant
  under twisting by $\chi_{-2}$. They are induced from the unramified
  quadratic extension.
\item The other $8$-dimensional one is induced from $\Q_2(\sqrt{3})$.
\item The two $10$-dimensional ones are principal series at $2$.
\item The $12$ dimensional one is also induced from $\Q_2(\sqrt{3})$.
  \end{itemize}
  Note that we obtain four Galois orbits of newforms from the field
  $\Q_2(\sqrt{-1})$, while we expect only two of them. This phenomenon
  seems to persist for higher weights. The value $\NCM(2^8)$ seems to
  be $12$ (we have computed up to weight 28), while our lower bound equals $10$.
\end{example}

It would
be interesting to have some statistical data on the size of the
smallest $k$ for equality to hold (which in particular is related to
an effective proof of Theorem \ref{thm:typesexistence}).

Note that a suitable variant of Question~\ref{conecture:maeda} makes
sense for general $N$. Giving a more involved formula (as the example
explained for level $N=11^2 \cdot 31^2$) obtained by a detailed study
of the local types for primes dividing $N$ (and the coefficient fields
of such modular forms) one can ask whether the obtained inequality is
best possible. The cases not covered by Theorem~\ref{cor:NCMlowbound}
involve very large levels, so we could not gather any computational
data which might suggest a positive or negative answer for the
generalized Maeda's problem on general levels $N$.

\section{Possible generalizations}

There are many similar situations to study. The first natural question
is what happens when working with modular forms with non-trivial
Nebentypus. The situation is more subtle, and there are two different
problems to be considered. One is that we are forced to look at
minimal twists (and we only considered minimal quadratic twists in the
trivial Nebentypus situation). The second one, is that there are no
Atkin-Lehner involutions! One needs to replace them by the operators
defined by Atkin and Li in \cite{Atkin-Li}. We will consider this
situation in a sequel of the present article. There is an obstacle in
studying the number of orbits of modular forms with Nebentypus coming
from its computational complexity. Still, it is true in this situation
that the number of CM modular forms is bounded in the weight.

A second reasonable generalization is to study the case of Hilbert
modular forms, i.e. changing the base field $\Q$ by a totally real one
$F$. To study CM modular forms, the same ideas in \cite{Tsaknias2012a}
give a bound of their number independent of their weight. The same
techniques developed in this article can be used to compute the number
of local types of level $\id{p}^n$, for $\id{p}$ a prime ideal. Still,
the formula is more involved in each case, as it depends on the degree
$[F:\Q]$, on the inertial degree of $\id{p}$ over $\id{p} \cap \Z$ and
its ramification degree. Then there are other invariants appearing
related to the class number of $F$. An interesting question to study
is if there are other type of Galois invariants besides the ones
described in this article and the ones coming from the class group (it
also seems natural, in the same way that CM forms were treated
separately in the present article, to treat separately special cases
of Hilbert modular forms such as those coming from base change, or
from base change up to twist, from a smaller field). The toy example
should be that of a real quadratic field, where base change forms are
easy to handle by the results of the present article.

At last, it is natural to consider a similar question for other
algebraic reductive groups $G$ over $\Q$ to see if there are more
invariants than those appearing for $\GL_2$.  For example if $G$ is
the group obtained from a rational quaternion algebra ramified at an
even number of finite places, by the Jacquet-Langlands correspondence,
automorphic forms for $G$ correspond to (some particular) automorphic
forms on $\GL_2$. In particular, all the results of this article work
for such algebraic groups, and we do not expect new invariants for
such groups (as we do not expect them for $\GL_2$).  As suggested to
the first author by M. Harris, it would also be interesting to test
other groups like $\GL_n(\A_\Q)$ (or over a totally real number field)
or $\GSp_n(\A_\Q)$ to see if these phenomena persist.  Again, in such a
context it seems natural to exclude all ``special'' forms (i.e., those
coming from automorphic forms from a smaller reductive group via
Langlands functoriality) before checking if there is uniqueness for
orbits with given local constraints, for sufficiently large weight
(existence results are known in this generality, as was mentioned in
section 4).  We must admit that we did not consider any of these
problems from a theoretical point of view nor gathered any
computational evidence but it is our hope that this article may spark
some research interest towards this direction.

\subsection{Applications of Question~\ref{conecture:maeda}}
It is well known that Maeda's conjecture has many applications to
different problems in number theory. The veracity of
Problem~\ref{conecture:maeda} has as many applications as the original
conjecture. Let us recall some of them.

\subsubsection{Inner twists} The truth of Question
\ref{conecture:maeda} implies that the existence of inner twists
for a newform is a purely local property, depending on the local
types of the form.

\begin{prop}
  Assume that Question~\ref{conecture:maeda} has an affirmative answer. Let
  $f \in S_{k}(\Gamma_0(p^r))$ be a newform of prime power level whose
  local type and Atkin-Lehner sign equals $(\tau, \epsilon)$. Let
  $\mu$ be an inner twist of $f$ (i.e. a finite order character such
  that $f \otimes \mu$ is Galois conjugate to $f$). Then for any
  $k'\geq k_0$ (where $k_0$ is the weight after which the conjecture
  becomes effective) and any newform $g \in S_{k'}(\Gamma_0(p^r))$
  whose local type equals $\tau$ and whose Atkin-Lehner sign equals
  $\epsilon$, $\mu$ is an inner twist of $g$.

  Furthermore, if $\mu$ is any finite order character ramified only at
  one prime $p$, for any pair $(\tau, \epsilon)$ as before invariant
  (up to Galois conjugation) under twisting by $\mu$ and every 
  $k'\geq k_0$, all newforms $g \in S_{k'}(\Gamma_0(p^r))$ with local
  data $(\tau,\epsilon)$ have inner twist given by $\mu$.
\end{prop}

\begin{proof}
  The proof is automatic due to the uniqueness result implied by
  Question~\ref{conecture:maeda} (see Remark \ref{remark:maeda}). If
  we assume that $f$ has an inner twist by $\mu$ this implies that
  $\mu$ ramifies only at $p$ and that the local type and Atkin-Lehner
  signs of $f$ are invariant (up to Galois conjugation) under twisting
  by $\mu$. Therefore the same is true for any $g$ with the same local
  data. Since Maeda's conjecture implies uniqueness of the Galois
  orbit with a fixed local data (at prime power level and weight
  greater or equal to $k_0$) the twist $g \otimes \mu$ must lie in the
  same orbit as $g$. The last claim follows from the same argument with the existence result given by Theorem~\ref{thm:typesexistence}.
\end{proof}

\subsubsection{Base Change}
The proof of (non-solvable) Base Change for classical modular forms
and other cases of Langlands functoriality given
in \cite{Dieulefait2012} relies on the construction of a ``safe'' chain
of congruences linking arbitrary pairs of modular Galois
representations. For the construction of such a chain in the
aforementioned article, it is crucial to “pass through” a space of
newforms having a unique Galois orbit: the space used in loc. cit. is
a space of forms of prime level with non-trivial Nebentypus of fixed
order and relatively large weight, a space that was computed in order
to check that it contains indeed a unique Galois orbit of
newforms. The conjecture proposed in Question~\ref{conecture:maeda} (i.e., the truth of
the claim stated therein), gives an alternative and more theoretical
way to complete the proof of Base Change (a proof not requiring
computations): in fact, for the construction of the “safe chain”
instead of a space with a unique Galois orbit (which is an option, but
requires non-trivial Nebentypus) it is enough to know that in certain
spaces of newforms of sufficiently large weight (and prime power
level) there is a unique orbit with a specific supercuspidal local
inertial type at the prime in the level, a fact that is implied by our
conjecture.

This strategy for the construction of safe chains is explained in
\cite{Rayuela}, in the more general context of Hilbert modular
forms over a given totally real number field $F$. The construction of
a safe chain connecting the Galois representations attached to any
pair of Hilbert newforms over $F$, from whose existence relative
non-solvable Base Change would follow immediately, can be reduced
following the strategy described in loc. cit. to a case where the two
Hilbert newforms have the same level, the same (large) parallel
weight, and common inertial types at primes in their common level,
thus a suitable generalization to Hilbert modular forms of the
uniqueness claim proposed in Question~\ref{conecture:maeda} gives a
way of completing the safe chain described in loc. cit. thus
completing the proof of relative Base Change.

\bibliography{biblio}
\bibliographystyle{alpha}

\end{document}